\documentclass[12pt]{amsart}

\usepackage{amssymb, amsmath,  amsthm}
\usepackage[latin1]{inputenc}
\usepackage{graphics,graphicx}
\graphicspath{{./figs/}}
 
\usepackage[final]{hyperref}
\usepackage{color}
\usepackage{algorithm}
\usepackage{algorithmicx}
\usepackage{algpseudocode}
\usepackage{ulem}
\usepackage{mathrsfs}
\usepackage{enumitem}
\newcommand\ds\displaystyle

\usepackage[a4paper, centering]{geometry}

\usepackage{color}
\usepackage{graphicx}
\graphicspath{{./pics/}}
\newtheorem{theorem}{Theorem}

\newtheorem{lemma}[theorem]{Lemma}

\theoremstyle{definition}

\theoremstyle{remark}
\newtheorem{remark}[theorem]{Remark}

\def\R{\mathbb{R}}

\definecolor{verde}{RGB}{20,150,100}

\newcommand{\Om}{\Omega}
\newcommand{\om}{\omega}

\newcommand{\ra}{\rightarrow}

\def \e{\varepsilon}

\newcommand{\vps}{\varepsilon}

\newcommand{\Hm}{{\mathcal H}^{N-1}}
\newcommand{\sm}{\setminus}

\newcommand{\sq}{\subseteq}

\def\G{\Gamma}

\newcommand \Rn{\mathbb{R}^N}
\newcommand \Hs{\mathscr{H}^{N-1}}

\newcommand \Per{\text{Per}}
\newcommand \eps{\epsilon}
\newcommand \A{\mathcal{A}}

\begin{document}

\title[]{Boundary behavior of Robin problems in non-smooth domains}

\author[D. Bucur]
{Dorin Bucur}
\address[Dorin Bucur]{ Univ. Savoie Mont Blanc, CNRS, LAMA \\
73000 Chamb\'ery, France
}
\email[D. Bucur]{dorin.bucur@univ-savoie.fr}

\author[A. Giacomini]
{Alessandro Giacomini}
\address[Alessandro Giacomini]{DICATAM, Sezione di Matematica, Universit\`a degli Studi di Brescia, Via Branze 43, 25133 Brescia, Italy}
\email[A. Giacomini]{alessandro.giacomini@unibs.it}
\author[M. Nahon]
{Mickael Nahon}
\address[Mickael Nahon]{ Univ. Savoie Mont Blanc, CNRS, LAMA \\
73000 Chamb\'ery, France
}
\email[M. Nahon]{mickael.nahon@univ-smb.fr}

\thanks{A.G. is also member of the Gruppo Nazionale per L'Analisi Matematica, la Probabilit\`a e loro Applicazioni (GNAMPA) of the Istituto Nazionale di Alta Matematica (INdAM)}

\subjclass[2010]{}

\date{\today}
\maketitle

\begin{abstract}   
We analyze strict positivity at the  boundary for nonnegative solutions of Robin problems in general (non-smooth) domains, e.g. open sets with rectifiable topological boundaries having finite Hausdorff measure. This question was raised  by Bass, Burdzy and Chen in 2008 for harmonic functions, in a probabilistic context. We give geometric conditions such that the  solutions of  Robin problems associated to general elliptic operators of $p$-Laplacian type, with a positive right hand side, are globally or locally bounded away from zero at the boundary. Our method, of variational  type,  relies on the analysis of an isoperimetric profile of the set and provides quantitative estimates as well.
\end{abstract}

\section{Introduction}
The purpose of this paper is to study the boundary behavior of the solution of a Robin  problem for the $p$-Laplace operator associated to some non-negative right hand side in a non-smooth domain. For instance, for the Laplace operator, if both the domain and the solution were smooth, a consequence of the boundary Hopf principle is that the solution is strictly positive on the boundary (and of course in the interior of the domain). In this paper, we are interested in this question in the context of non smooth boundaries where the strong form of the boundary condition may not hold.

To introduce the problem, let $\Om\sq \R^N$ be an open, connected, bounded set with a boundary which, for the moment and for expository reasons, is assumed to be  of class $C^2$ (this regularity assumption will be removed later). Given a parameter $\beta >0$ and  a smooth function  $f:\Om\ra \R$, we consider the problem

\begin{equation}
\label{pde}
\begin{cases}
-\Delta u=f&\hbox{in }\Om\\
\ds \frac{\partial u}{\partial \nu}+\beta u=0&\hbox{on }\partial\Om,
\end{cases}
\end{equation}
$\nu$ being the outer normal at the boundary. Under the hypothesis $f \ge 0$ in $\Om$ one gets from the maximum principle that $u \ge 0$ in $\Om$ and, if $f \not \equiv 0$, $u$ is strictly positive inside $\Om$. The minimum of $u$ has to be searched on  the boundary of $\Om $ where, as a consequence of the Hopf lemma,  the normal derivative has to be strictly negative. Consequently, from the strong form of the  Robin boundary condition  $\frac{\partial u}{\partial n}+\beta u=0$ one gets that $u$ has to be strictly positive at its minimum. Finally,  there exists some $\delta >0$ such that
\begin{equation}\label{br01}
\forall x \in \Om, \;\; u(x) \ge \delta >0.
\end{equation}

If $\Om$ is not of class $C^2$ (for instance it is Lipschitz, or less regular but  smooth  enough to give sense to a weak form of the problem), then the Hopf lemma can not be anymore used to arrive to the same conclusion. The normal direction at the boundary might not be properly defined at some specific point of the boundary, and the strong pointwise form of the Robin boundary conditions may not apply. Nevertheless, lower bounds as \eqref{br01} may hold provided that, intuitively, there is no "too high" concentration of the boundary around one point.

Looking for strict positive estimates similar to \eqref{br01} in non smooth domains,  the first result is due to Bass, Burdzy and Chen \cite{BBC08}, by probabilistic methods in the context of harmonic functions. They identify a class of nonsmooth domains, including the Lipschitz ones, for which \eqref{br01}  holds for non-negative harmonic functions solving a Robin problem (see also \cite{BL14,CK16} for a similar question in the context of free discontinuity problems). The result of  \cite{BBC08} relies on a probabilistic method and applies to the Laplace operator. The key argument strongly uses the linearity of the equation,  involving the study of Green functions and the use of a uniform boundary Harnack inequality. In particular, the geometric conditions given in  \cite{BBC08}   require that   the domain should be written as union of sets with Lipschitz boundaries sharing the same  bound on the Lipschitz norm, up to a set of  $\Hm-$ Hausdorff measure equal to zero. This covers the case of  Lipschitz domains, of some domains with cusps and of some fractal domains.

 Two more  recent results by Gesztesy, Mitrea and Nichols \cite{gmn14} on the one hand, and by Arendt and ter Elst \cite{AtE17} on the other, show that a first non-negative eigenfunction of the Robin Laplacian in a Lipschitz set $\Om$  satisfies the bound from below \eqref{br01}, as soon as it is is not identically equal to $0$. Their arguments are based on the analysis of related semigroups acting on $C(\partial \Om)$ and are a consequence of a regularity property of the eigenfunctions of the Robin Laplacian in Lipschitz domains,  in particular their continuity up to the boundary.

The purpose of this paper is to analyze  the boundary strict positivity inequality \eqref{br01} in a more general context of non-smooth domains and of nonlinear PDEs of $p$-Laplacian type (even if they are not of energy type). Our technique is based on variational arguments, it allows to handle global and local results and to give quantitative estimates of the lowest value in terms of some average sum of the solution. The key ingredient is the behaviour of a kind of global (or local) isoperimetric profile of the set, which depends on $p$.

We deal with  bounded, open, connected sets with a rectifiable topological boundary which have finite $(N-1)$-Hausdorff measure.  In this case, traces of $W^{1,p}$-Sobolev functions are $\Hm$-pointwise well defined  on $\partial \Om$ and the Robin problem is well posed in a weak form for operators of $p$-Laplacian type.
The geometric properties which play a crucial role in the validity of \eqref{br01} can, in some situations, be related to the local control of the  $L^1$ norm of the trace of an BV function on $\partial \Om$ and to local reinforced isoperimetric inequalites  via the summability of an isoperimetric profile function. 
 Of course, a class of domains which satisfy naturally these geometric properties includes all Lipschitz sets. Nevertheless,  Lipschitz regularity is not, in general, required for the property to hold. 
 Our analysis  provides  a quantitative estimate for the constant $\delta$ in \eqref{br01} in terms of the geometry and of the mass of the solution $u$ on low sublevels.

For simplicity, we focus on the $p$-Laplacian equation which obeys an energetic variational principle. It turns out that the energetic principle is useful for the comprehension of the existence of a solution in a nonsmooth domain, but it is not a key tool for the boundary behaviour. We  discuss in the last section how the results for the $p$-Laplacian extend (under the assumption that a weak solution exists) to  more general monotone elliptic operators in divergence form, which are not necessarily related to energy minimization.

\section{Global boundary behavior}
   Let    $\Om \sq \R^N$ be a bounded, open, connected set such that $\partial \Om$ is rectifiable and $\Hm(\partial \Om) < +\infty$.  Let $\beta >0$ be a constant,  $p \in (1, +\infty)$, $p'=\frac{p}{p-1}$. Let $f \in L^{p'} (\Om)$, $f \ge 0$ not identically equal to $0$.

Let $v\in W^{1,p}(\Om)$. Then for $\Hm$-a.e. point $x \in \partial \Om$ the function $v$ extended by $0$ outside $\Om$ has  upper and lower approximate limits, denoted $v_+$ and $v_-$.  Moreover, 
$$v \to \int_{\partial\Om} (|v_+|^p+|v_-|^p) d \Hm$$
 is lower semicontinuous for the weak topology in $W^{1,p}(\Om)$. There exists a  function $u$ which minimises in $W^{1,p}(\Om)$ the following energy
\begin{equation}\label{bgnr03}
E(v)=\frac{1}{p}\left(\int_{\Om}|\nabla v|^pdx+\beta \int_{\partial\Om} (|v_+|^p+|v_-|^p) d \Hm \right)-\int_{\Om}fvdx,
\end{equation}
 and the minimizer is unique, from the strict convexity of $E$. For the existence part, the key results are the Poincar\'e inequality with trace term which occurs in this setting together with the compactness result in $L^p(\Om)$ which holds for a sequence of functions with bounded energy. For all these results we refer to \cite{AFP} and more specifically to \cite{BG16}. 
Alternatively, relying on the space of special functions with bounded variation $SBV^{1/p}(\Rn, \R^+) $ (see \cite{AFP}), this procedure to define weak solutions can be seen by the minimization of $E$ over   $SBV (\Rn, \R^+) $-functions with support in $\Om$ and jump in $\partial\Om$ (see  \cite{BG16}). As the existence part is not relevant for our purpose in this paper, we shall not detail it. 
 
 If $\Om$ is Lipschitz, then 
 $u$ is the (weak) variational solution of

\[\begin{cases}
-\Delta_pu=f &\mbox{ in } \Omega,\\ |\nabla u|^{p-2}\frac{\partial u}{\partial \nu}+\beta |u|^{p-2}u=0 &\mbox{ on }  \partial\Omega,
\end{cases}\]
meaning that $u \in W^{1,p}(\Om)$ and for every $v \in W^{1,p}(\Om)$
$$
\int_\Om |\nabla u|^{p-2} \nabla u \nabla v dx + \beta \int_{\partial \Om} |u|^{p-2}u v d \Hm= \int_\Om fv dx.$$
If $\Om$ has rectifiable boundary with finite Hausdorff measure, the weak solution $u \in W^{1,p}(\Om)$ satisfies for every $v \in W^{1,p}(\Om)$
\[ \int_\Om |\nabla u|^{p-2} \nabla u \nabla v dx + \beta \int_{\partial \Om} \left(|u^+|^{p-2}u^+v^+ +|u^-|^{p-2}u^-v^- \right)d \Hm=\int_\Om fv dx,\]
where $u^+(x),u^-(x)$ (and $v^+(x),v^-(x)$) denote the traces of $u$ at $x\in\partial\Om$, ordered by the choice of a normal vector normal to $\partial\Om$ at $x$ which is well-defined $\Hm$-almost everywhere.

For a measurable set with finite perimeter $\om\sq \Om$, we denote  $\partial ^* \om$ its reduced boundary and
$$
P^e(\om ):= \int_{\partial \Om}\left( 1_\om^++1_\om^-\right) d\Hm\qquad\text{and}\qquad  P^i(\om):= P(\om,\Om),
$$
the exterior  and the interior perimeter of $\om$, respectively. Note that the exterior perimeter counts the both sides of a crack of $\partial \Om$ which crosses $\om$. In particular, $P^e (\Om)$ equals the sum between the Hausdorff measure of the topological boundary points of density $\frac{1}{2}$ in $\Om$ and twice the Hausdorff measure of topological boundary points of density $1$. For simplicity, we denote $\partial ^i \om= \partial ^* \om \cap \Om$ and $\partial ^e \Om= \partial ^* \om \cap \partial \Om$.

\smallbreak
We define the $p$-isoperimetric profile of $\Om$ as follows
$$I : \left(0, \frac 12 P^e (\Om)^\frac 1 p |\Om| ^\frac{1}{p'}\right) \ra \R\cup\{+\infty\},$$
\[I (m)=\inf \left\{P^i(\om): \ \om\subset\Om, |\om|\le \frac{|\Om|}{2}, m \le P^e (\om )^\frac{1}{p}|\om|^{\frac{1}{p'}}\right\}.\]
The value of $I  (m)$ is well defined for every $ 0<m\le  \frac 12 P^e (\Om)^\frac 1p |\Om| ^\frac{1}{p'}$ since the set above is not empty.  Moreover, the function $I$ is non decreasing. Note that the main interest is related to in the behavior of $I$ when $m$ is small.

\begin{theorem}\label{bgnr01}
Assume $\frac{1}{I }$ is summable. Then 
\[\inf_{x\in \Om} u(x) \geq\frac{T}{2}>0,\]
where $T>0$ denotes the largest number such that
\begin{equation}
\label{eq:adm1}
|\{u<T\}| \leq \frac {1}{ 2} |\Om| 
\end{equation}
and
\begin{equation}
\label{eq:adm2}
\beta^{\frac{1}{p}}\int_{0}^{P^e(\Om)^{\frac{1}{p}}|\{u<T\}| ^ \frac {1}{{p'}} }\frac{dm}{I  (m)}\leq \frac{1}{2}.
\end{equation}
\end{theorem} 
Before the proof, let us make some remarks. The summability condition on the isoperimetric profile seems abstract and difficult to handle. However, as we shall point out in some examples, there are situations where efficient estimates of   $I $ can be obtained under controllable geometric assumptions, for instance in the case of cusps. Note as well that  $T$ is indeed strictly positive, otherwise the solution $u$ would vanish on a set of positive measure, which is impossible for a nonvanishing, nonnegative, $p$-superharmonic function (\cite[Theorem 7.12]{HKM93}). 

Let also point out that a similar result could be obtained for a nonconstant $\beta \in L^1(\partial \Omega, \R^+)$.  However, in this case, a fine study would require the analysis of an isoperimetric profile involving the function $\beta$, which is not easy in practical situations. In the particular case in which the function $\beta$ is bounded, our result applies to the constant $\|\beta\|_\infty$. In the next section, where we perform a local analysis,  considering such functions $\beta$ which are locally bounded may be of more interest. 
\begin{proof}
We note first that $u \ge 0$.
For every $t >0$, let $\om_t=\{u<t\}$ and  $g(t)=\int_{\om_t}|\nabla u|$. From the minimizing property of $u$ compared to $\max(u,t)$, we have the estimate
\begin{equation}\label{eq3.1}
\int_{\om_t}|\nabla u|^p \leq \beta t^pP^e(\om_t).
\end{equation}
Consequently,   H\"older inequality gives
\begin{equation}\label{eq3}
g(t)\leq t(\beta P^e(\om_t ))^\frac{1}{p}|\om_t|^{\frac{1}{p'}}.
\end{equation}

We denote $t_0=\sup\{s\geq 0:g(s)=0\}=\inf(u)$, our goal being to estimate $t_0$ from below. 
  If $t_0\le t\le T$, thanks to \eqref{eq:adm1} and to the monotonicity of $I$ we get
 \[I 
 \left(\frac{g(t)}{\beta ^\frac{1}{p} t}\right)\leq I  \left(P^e(\om_t )^\frac{1}{p}|\om_t|^{\frac{1}{p'}}\right)\leq P^i(\om_t )=g'(t).
 \]
Summing from $t_0$ to  $t$, we obtain (using again the monotonicity of $I$)
\[
 \int_{t_0}^{t}\frac{g'(\tau)d\tau}{I(g(\tau)/\beta ^\frac{1}{p}t)}\ge \int_{t_0}^{t}\frac{g'(\tau)d\tau}{I(g(\tau)/\beta ^\frac{1}{p}\tau)}\geq t-t_0,
\]
and the change of variable $m=g(\tau)/\beta ^\frac{1}{p}t$, the inequality becomes
\[ \beta ^\frac{1}{p}t\int_{0}^{g(t)/\beta ^\frac{1}{p}t}\frac{dm}{I( m)}\geq t-t_0.\]
With the initial estimate \eqref{eq3} and  equation \eqref{eq:adm2} for $T$, we finally obtain
\[1-\frac{t_0}{T}\leq \int_{0}^{P^e(\Om )^{\frac{1}{p}}|\om_T|^{\frac{1}{p'}}}\frac{dm}{I(m)}\leq \frac{1}{2},\]
so that $t_0\geq \frac{T}{2}$.
\end{proof}

\subsection{Examples}
To be more precise,  we discuss now some global geometric assumptions  on  $\Om$ which can be verified and which lead to strict positivity. We denote by $\partial \Om$ the topological boundary and by $\partial ^* \Om$ the measure theoretical boundary of $\Om$.

\medskip
\noindent {\bf Geometric interpretation.}  Let $\Om \sq \R^N$ be a bounded, open, connected set such that $\partial \Om$ is rectifiable and $\Hm(\partial \Om) < +\infty$.  Below, the constant $G>0$ may change from line to line. Assume that that there exist constants $G>0$ and $1\ge a,b\ge 0$   such that for all $\om \subset \Om$ with finite perimeter satisfying $|\om|\le \frac{|\Om|}{2}\wedge 1$ we have
\begin{equation}\label{br02.a}
P^e(\om ) \le G \frac{P^i(\om)}{|\om|^a}.
\end{equation}
\begin{equation}\label{br02.b}
|\om|^\frac{N-1}{N}\le G \left(P^i(\om)+P^e(\om)\right)\left(\frac{P^i(\om)}{P^e(\om)}\right)^b.\end{equation}

The first inequality gives a control  for the ratio between the boundary and inner perimeters by a constant which can blow up when the measure of the domain is small, while the second inequality can be interpreted as an improved isoperimetric inequality for domains with small measure with a large isoperimetric quotient. Then, Theorem \ref{bgnr01} applies for certain values of $a,b$, provided \eqref{br02.a} and \eqref{br02.b} are satisfied. We point out the following examples.

\begin{itemize}

\item In general, if 
\begin{equation}\label{bgnr02}
M:=\frac 1p+\frac {\frac {1}{p'} -\frac ap}{\frac{N-1}{N} +a (1-b)}>1,
\end{equation}
then the  integrability hypothesis of Theorem \ref{bgnr01} is satisfied and  strict positivity occurs. 
Indeed, assume that \eqref{bgnr02} occurs. We  study the integrability of the $p$-isoperimetric profile of $\Om$. 
Take $\om\sq \Om$ such that 
$$ m \le   P^e(\om)^\frac{1}{p}|\om|^{\frac{1}{p'}}.$$
We deduce from \eqref{br02.a} and \eqref{br02.b} a lower bound on $P^i(\om)$, which implies a lower bound on $I(m)$ by taking the infinimum among all admissible $\om$.\newline

If $P^e(\om)\leq P^i(\om)$, then the usual isoperimetric inequality applies and so there is a constant $c_N>0$ such that $|\om|^\frac{N-1}{N}\leq c_N \mathrm{Per}(\om)\leq 2c_NP^i(\om)$, so
\[P^i(\om)=P^i(\om)^{\frac{Np-N}{Np-1}}P^{i}(\om)^{\frac{N-1}{Np-1}}\geq A_N|\om|^\frac{(N-1)(p-1)}{Np-1}P^{i}(\om)^{\frac{N-1}{Np-1}}\geq A_N m^\frac{Np-p}{Np-1}\]
for some $A_N>0$. On the contrary if $P^i(\om)\leq P^e(\om)$ then we have
\begin{align*}
P^i(\om)^b&\geq c|\om|^\frac{N-1}{N}P^e(\om)^{b-1}\\
P^i(\om)&\geq c|\om|^aP^e(\om)\\
\end{align*}
for some $c>0$, so combining those as previously we obtain
\[P^i(\om)^M=P^i(\om)^{b\frac{1-\frac{1+a}{p}}{ \frac{N-1}{N}+a(1-b)}}P^i(\om)^{\frac{1}{p}+\frac{\left(1-\frac{1+a}{p}\right)(1-b)}{ \frac{N-1}{N}+a(1-b)}}\geq c|\om|^{ \frac{1}{p'} }P^e(\om)^\frac{1}{p} \ge m\]
In conclusion
$$I (m) \ge c\left(m^\frac 1M\wedge m^\frac{Np-p}{Np-1}\right).$$
Hence $\int_0^{m_0}\frac{1}{I (m)}dm <+\infty.$
\vskip10pt
\item  For $N=2$, the case $b= \frac 12$, $a= \frac{\alpha -1}{\alpha+1}$, $\alpha \in (1,2)$ covers the cuspidal domains in $\R^2$ $\{(x,y) : 1>x>0, |y|< x^\alpha\}$. The proof it is not direct, we refer to Section \ref{s.cusps} for a general approach of cusps. For the Laplace operator, the case $a=\frac 13$ and $b=\frac 12$ is critical as  $M$ defined  above  equals to $1$, and the positivity property does not hold (see \cite{BBC08}).
\vskip10pt
\item For a Lipschitz set $\Om$,  inequalities \eqref{br02.a}-\eqref{br02.b} are satisfied with $ a=b=0$.  
 Indeed \eqref{br02.b} reduces to the isoperimetric inequality. Concerning \eqref{br02.a}, the trace theorem in $BV$ applied to $1_\om$ together with the relative isoperimetric inequality in a Lispchitz domain (recall $|\om|\le \frac{1}{2}|\Om|$) yield
$$
P^e(\om)\le C_\Om \left[ P^i(\om)+|\om|\right]\le C_\Om \left[ P^i(\om)+|\Om|^{\frac{1}{N}}|\om|^{\frac{N-1}{N}}\right]\le \hat C_\Om P^i(\om),
$$
for some constants $C_\Om,\hat C_\Om>0$.

\end{itemize}

\begin{remark} \rm
More generally, we notice that inequality \eqref{br02.a} entails an inner density estimate for the set $\Om$ in the case $a=0$. Indeed, taking a boundary point $x_0\in \partial \Om$, and $A= B(x_0, t)$, from the isoperimetric inequality and  \eqref{br02.a} we get by the co-area formula that for a.e. $t\in (0, t_0)$ that
$$m(t)^\frac{N-1}{N} \le C_N (1+G) m'(t).$$
Above, $t_0$ is the radius of the ball of the same measure as $\frac{|\Om|}{2}$ and $m(t)= |\Om \cap B(x_0, t)|$. Then, by summing from $0$ to $r$, with $r\le t_0$, one gets
$$|\Om \cap B(x_0, r)| \ge \frac {r^N}{\big (NC_N (1+G)\big )^N},$$
giving the uniform inner density estimate. More generally $a\in [0,1/N)$ this implies an estimate
$$|\Om \cap B(x_0, r)| \ge cr^{\frac{N}{1-N\alpha}}.$$

\medskip

If $a=0$, inequality \eqref{br02.a} is in fact related to the BV-trace theorem of Anzellotti and Giaquinta, for which we refer the reader to \cite{ag78}. Indeed, if the topological boundary of $\Om$ coincides $\Hm$-a.e. with its reduced boundary, the existence of a constant $C_1$ above is related to the following inequality
$$C\int_{\partial \Om} |u| dx \le \int_\Om |Du| + \int_\Om |u| dx,$$
by $|D u|$ denoting the total variation measure associated to the BV function $u$. If $\partial \Om$ differs from $\partial ^* \Om$ by a set of positive $\Hm$-measure (for instance $\Om$ has cracks), then the result of Anzellotti and Giaquinta does not apply directly. It is not our goal here to analyze the existence of continuous traces. 
\end{remark}

\begin{remark}
{\rm
Let us sketch briefly a direct argument to prove that $\inf_\Om u>0$ in the Lipschitz case, which employes the standard relative isoperimetric inequality and the trace operator in $BV$ (see the considerations above). Using the same notation as in the proof of Theorem \ref{bgnr01}, let
$$
\omega_t:=\{x\in \Om\,:\, u(x)<t\}
$$
and let us assume by contradiction that $\omega_t\not=\emptyset$ for every $t>0$. Below $C$ indicates a constant depending only on $\Om$ which can very from line to line. The Lipschitz regularity of the boundary entails for $t$ small
$$
P^e(\om_t)\le CP^i(\om_t) \qquad \text{and}\qquad |\om_t|\le C P^i(\om_t)^{\frac{N}{N-1}},
$$
so that the comparison between the functions $u$ and $u\wedge t$ leads to
\begin{equation}
\label{eq:r1}
\int_{\om_t}|\nabla u|\,dx\le \left( \int_{\om_t}|\nabla u|^p\,dx \right)^{\frac{1}{p}}|\om_t|^{\frac{1}{p'}}\le [\beta t^p P^e(\om_t)]^{\frac{1}{p}}|\om_t|^{\frac{1}{p'}}
\le C tP^i(\om_t)^{1+\e},
\end{equation}
where $\e>0$. Setting $h(t):=P^i(\om_t)$, the coarea formula yields
$$
\int_{\om_t}|\nabla u|\,dx=\int_{0}^t h(\tau)\,d\tau=:H(t),
$$
so that
$$
\frac{1}{t^{1-\eta}}\le C\frac{H'(t)}{H(t)^{1-\eta}},
$$
where $0<\eta<1$.
Summing from $0$ to $t$ we get
\begin{equation}
\label{eq:r12}
t\le C \int_{0}^t h(\tau)\,d\tau.
\end{equation}
Coming back to \eqref{eq:r1}, we may write
$$
\frac{C}{t^p |\omega_t|^{\frac{p}{p'}}}\le \frac{h(t)}{\left( \int_{0}^t h(\tau)\,d\tau\right)^p},
$$
so that, using the monotony of $|\omega_t|$ and integrating on $[t_1,t_2]$
$$
\frac{C(t_2-t_1)}{t_2^p |\omega_{t_2}|^{\frac{p}{p'}}}\le \frac{1}{p-1}\frac{1}{\left( \int_{0}^{t_1}h(\tau)\,d\tau\right)^{p-1}}.
$$
Setting $t_2=2t_1$ we obtain
$$
\int_0^{t_1}h(\tau)\,d\tau\le C t_1|\omega_{2t_1}|.
$$
In view of \eqref{eq:r12} we get
$$
t_1\le C\int_0^{t_1}h(\tau)\,d\tau\le Ct_1|\omega_{t_1}|
$$
which is false if $t_1$ is small enough, so that the result follows.
\par
A similar reasoning can be used starting from the more general inequalities \eqref{br02.a} and \eqref{br02.b}, but the approach through the isoperimetric profile $I(m)$ encompasses all the situations in a very elegant way.
}
\end{remark}

\section{Local boundary behaviour}

In this section, we give a localized positivity result. In particular, this  applies to the case in which $\beta$ is a function, locally bounded from above. In the computations below we assume without loosing generality that $\beta\le1$.

 Under the previous hypotheses on $\Om$, let $\Gamma$ be a relatively open subset $\Gamma\subset \partial\Om$. For a measurable set $\om \sq \Om$ with finite perimeter, we denote the exterior/interior perimeters relative to $\Gamma$ 
\begin{align*}
P_\G^e(\om)&
=\int_{\Gamma} (1_\om^++1_\om^-) d\Hm,\ P_\G^i(\om)
= \Hm (\Om \cap J_{1_\om} )+ \int_{\partial \Om \sm \Gamma} (1_\om^++1_\om^-) d\Hm,
\end{align*}
and the associated isoperimetric profile  $I_\G: (0, \frac 12 P^e_{ \Gamma}(\Om )^\frac 1p |\Om| ^\frac{1}{p'}]\ra \R \cup\{+\infty\}$
\begin{align*}
I_\G(m)&=\inf\left\{P_\G^i(\om),\om\subset\Om\text{ s.t. }|\om|\leq \frac{1}{2}|\Om|,\ P_\G^e(\om)^\frac{1}{p} |\om|^{ \frac{1}{p'} }\ge m\right\}.\\
\end{align*}

We start with a general result, which applies, in particular, to the  minimizers  of \eqref{bgnr03}, as soon as $f \ge 0$, $f \not=0$ and $\beta \le 1$. 

\begin{theorem}\label{bgnr05}
Let $u\in W^{1,p} (\Om)$, $u \not= 0$,  be a non negative function such that for any $v\in W^{1,p}(\Om)$ with $v\geq u$, $\text{dist}(\{v>u\},\partial\Om\setminus\Gamma)>0$,
\begin{equation}\label{bgnr04}
 \int_{\Om}|\nabla u|^p\, dx\leq \int_{\Om}|\nabla v|^p\, dx+  \int_{\Gamma} [v_+^p+v_-^p]\,d\Hm.
 \end{equation}
Suppose also that $1/I_{ \Gamma }$ is summable in a neigbourhood of the origin. Then for any compact set $K$ in $\Gamma$,
\[\underset{z(\in \Om)\to K}{\mathrm{ess{-}liminf}}u(z)>0,\]
meaning that there exist $r,\eps>0$ such that \[\left|\{u<\eps\}\cap \{z\in\Om:\text{dist}(z,K)<r\}\right|=0.\]
\end{theorem}

We start with  a technical result.
\begin{lemma}
Let $I_\G$ be as defined above and let $J_{ \Gamma }$ be defined by 
$$J_\Gamma(m):=\inf\left\{P_\G^i(\om),\om\subset\Om\text{ s.t. }|\om|\leq \frac{1}{2}|\Om|,\ P_\G^e(\om)^\frac{1}{p}|\om|^{ \frac{1}{p'} }+|\om|\ge m\right\}.$$
Then $1/I_\G$ is summable if and only if $1/J$ is summable.
\end{lemma}
\begin{proof}
Let $m>0$. Clearly $I_\G(m)\geq J_\Gamma(m)$. We claim that there exist $C,\vps>0$, that depend only on $N$ and $|\Om|$, such that
\begin{equation}
\label{eq:claim1}
J_\Gamma(m)\geq I_\G(\vps m)\wedge Cm^{\frac{N-1}{N}}.
\end{equation}

Then since 
\[\frac{1}{I_\G(m)}\leq \frac{1}{J_\Gamma(m)} \leq \frac{1}{I_\G(\vps m)}+\frac{1}{Cm^{\frac{N-1}{N}}},\]
the conclusion follows.
\par
Let us check claim \eqref{eq:claim1}. Indeed, let $\om\subset \Om$ be a domain such that $|\om|\leq \frac{1}{2}|\Om|$, and $|\om|+|\om|^{ \frac{1}{p'}}P_\G^e(\om)^\frac{1}{p}\geq m$. Let $\eta>0$ that will be fixed later, and suppose that 
\[P_\G^e(\om)\leq \eta|\om|.\]
Then, by the classical isoperimetric inequality, there is a constant $a_N>0$ such that
\[a_N|\om|^{\frac{N-1}{N}}\leq P_\G^i(\om)+P_\G^e(\om)\leq P_\G^i(\om)+\eta |\om|^{\frac{N-1}{N}}|\Om|^\frac{1}{N}.\]
We fix $\eta:=\frac{a_N}{2|\Om|^\frac{1}{N}}$. Then
\[P_\G^i(\om)\geq \frac{a_N}{2}|\om|^\frac{N-1}{N}\geq \frac{a_N}{2(1+\eta^\frac{1}{p})^\frac{N-1}{N}}m^\frac{N-1}{N}.\]
On the other hand, suppose that $P_\G^e(\om)> \eta |\om|$. Then
\[m\leq |\om|^{ \frac{1}{p'} }P_\G^e(\om)^\frac{1}{p}+|\om|\leq (\eta^{-\frac{1}{p}}+1)|\om|^{ \frac{1}{p'} }P_\G^e(\om)^\frac{1}{p},\]
so 
$$
P_\G^i(\om)\geq I_{ \Gamma}\left(\frac{m}{\eta^{-\frac{1}{p}}+1}\right).
$$ 

Claim \eqref{eq:claim1} follows by choosing $\vps:=\frac{1}{\eta^{-\frac{1}{p}}+1}$ and $C:=\frac{a_N}{2(1+\eta^\frac{1}{p})^\frac{N-1}{N}}$.

\end{proof}

\begin{proof}[Proof of Theorem \ref{bgnr05}]  Since $K$ is compact, let $ x_0 \in K$ and $r>0$ be small enough such that 
$$
 r<dist(x_0, \partial\Om \setminus \Gamma) \qquad\text{and}\qquad |B_{x_0,r}\cap\Om|\leq \frac{1}{2}|\Om|.
$$ 
It is enough (by compactness) to prove that $u$ is essentially bounded below on a smaller ball $B_{x_0,r'}\cap\Om$ for some $r'<r$. We proceed by contradiction and suppose without loss of generality that 
$$
\underset{x(\in\Om)\to x_0}{\mathrm{ess{-}liminf}} u(x)=0.
$$
Let $\gamma>0$ that will be fixed small enough at the end of the proof. For any $t\in (0,r\gamma)$ we set
\[
\om_t:=\{x\in \Om:u(x)<t-\gamma|x-x_0|\}
\]
and 
$$
g(t):=\int_{\om_t}|\nabla (u(x)+\gamma|x-x_0|)|\,dx.
$$
Notice that $\om_t\subset B_{x_0,r}$. By our hypothesis, $\om_t$ has positive measure for all $t\in (0,r\gamma)$, and $g(t)=0$ for some $t>0$ implies that $x\mapsto u(x)+\gamma |x-x_0|$ would be locally constant in $\om_t$, which is at most true for one $\gamma$, thus we suppose that $g(t)>0$ for all $t\in (0,r\gamma)$. 
\par
Testing $u$ against $\max(u,t-\gamma|\cdot-x_0|)$ gives
\[\int_{\om_t}|\nabla u|^p\,dx\leq \gamma^p|\om_t |+t^p  P_\G^e(\om_t ),\]
 which can  be simplified further into
\[
\int_{\om_t }|\nabla u|^p\,dx\leq \gamma^p|\om_t |+r^p\gamma^p P_\G^e(\om_t ).
\]
Now,
\[g (t)\leq \gamma |\om_t |+|\om_t |^{ \frac{1}{p'} }\left(\int_{\om_t }|\nabla u|^p\,dx\right)^\frac{1}{p}\leq 2\gamma |\om_t |+r\gamma|\om_t |^{ \frac{1}{p'} }  P_\G^{e}(\om_t )^\frac{1}{p},\]
so in particular $\frac{g(t)}{(2+r)\gamma}\leq |\om_t |+|\om_t |^{ \frac{1}{p'} }P_\G^{e}(\om_t )^\frac{1}{p}$. By monotonicity of $J_\Gamma$,  and since $\om_t$ is at positive distance from $\partial\Om\setminus\Gamma$  we get
\[J_\Gamma\left(\frac{g(t)}{(2+r)\gamma}\right)\leq P_\G^i(\om_t )= g'(t).\]
Summing from $t=0$ to $r\gamma$ we obtain
\[r\gamma\leq\int_{0}^{r\gamma}\frac{ g'(t)dt}{J_\Gamma\left(\frac{g(t)}{(2+r)\gamma}\right)}=(2+r)\gamma \int_{0}^{\frac{g(r\gamma)}{(2+r)\gamma}}\frac{dm}{J_\Gamma(m)}.\]
Using our previous estimate on $g$, we obtain
\[\int_{0}^{|\om_{r\gamma} |+|\om_{r\gamma} |^{ \frac{1}{p'} } P_\G^e(\om_{r\gamma} )^\frac{1}{p}}\frac{dm}{J_\Gamma(m)}\geq \frac{r}{r+2}.\]
In particular this means that $|\om_{r\gamma} |+|\om_{r\gamma} |^{ \frac{1}{p'} }  P_\G^e(\om_{r\gamma} )^\frac{1}{p}$ is bounded below by a constant that does not depend on $\gamma$. However $P_\G^e(\om_{r\gamma} )\leq  2\Hm(\Gamma)<\infty$ and $|\om_{r\gamma} |\leq |B_{x_0,r}\cap \{u< r\gamma\}|\underset{\gamma\to 0}{\longrightarrow}0$, which is a contradiction.
\end{proof}

\section{Analysis of cusps}\label{s.cusps}
We focus in this section on the behaviour of the solution of the Robin problem in a cusp in $\R^N$. 
We consider $h\in\mathcal{C}^1(\R_+,\R_+)$, an increasing function such that $h(0)=0$, $h'(0)=0$, $\sup_{\R_+}|h'|<\infty$ and assume that $h^{N-1}$ is convex.
Let 
$$
\Om:=\{x=(x_1,x')\in\R\times\R^{N-1}\text{ s.t. }\ x_1>0,\ |x'|<h(x_1)\},\\
$$
 and for $t\in (0,1)$ let us set
$$
\Om_t:=\Om\cap \{(x_1,x')\in\Om:x_1\leq t\}.
$$

We introduce the isoperimetric profile of $\Om$ restricted to revolution sets by
\begin{equation}
\label{eq:defIsym}
I_{sym}(m)=\inf\left\{P^i(\om)\,:\,\ \om \in \A \text{  and }  m\le P^e(\om)^{\frac{1}{p}}|\om|^{ \frac{1}{p'}}\right\},
\end{equation}
where
$$
\mathcal{A}:=\left\{A\subset\Om\text{ open, rotationally symmetric  with respect to the $x_1$-axes, and } |A|<+\infty \right\}.
$$
 The following technical result, whose proof will be given at the end of the section, will be essential to recover the summability property related to the isoperimetric profile of $\Om_t$, which is the key ingredient of our approach. Of course, the summability  depends on the behaviour near the origin of the function $h$ which defines the cusp.

\begin{lemma}
\label{lem:Isym}
 We have
\[
\label{eq:ineqL5}
\liminf_{t\to 0}\frac{I_{sym}\left(2P^e(\Om_t)^\frac{1}{p}|\Om_t|^{\frac{1}{p'}}\right)}{P^i(\Om_t)}>0.
\]
\end{lemma}

The main result of the section  is the following estimate near the cusp of $\Om$. We fix $t=1$.

\begin{theorem}
\label{thm:t-cusp}
Let   $f \in L^{p'}  (\Om_1)$, $f \ge 0$, $f  \not\equiv  0$ and let $u\in W^{1,p}(\Om_1)$ be a  minimizer in $W^{1,p}(\Om_1)$ of
\[E(v)=\frac{1}{p}\left(\int_{\Om_1}|\nabla v|^pdx+\beta \int_{\partial^e\Om_1}v^p d \Hm \right)-\int_{\Om_1}fvdx,
\]
under the constraint $v= 1 $ on $\partial^i\Om_1$. Then
\[
\left(\int_{0}^{1}\left(\frac{\int_0^t h^{N-2}\, dx_1 }{\int_0^t h^{N-1}\, dx_1 }\right)^{\frac{1}{p}}dt<\infty\right)\Rightarrow\left(\inf_{\Om_{1/2}}u>0\right).
\]
\end{theorem}

\begin{proof}
Let $I_{sym}$ be the restricted isoperimetric profile associated to $h$ as defined  in \eqref{eq:defIsym}.  Note that for $t\in (0,1)$ we have
\begin{align*}
|\Om_t|&=\alpha_{N-1}\int_{0}^{t}h^{N-1}\, dx_1 ,\\
P^e(\Om_t)&= N\alpha_{N-1}\int_{0}^{t}h^{N-2}\sqrt{1+h'^2}\, dx_1 ,\\
P^i(\Om_t)&=\alpha_{N-1}  h^{N-1}(t).
\end{align*}
So according to the estimate on $I_{sym}$  given by Lemma \ref{lem:Isym} , and since $|h'|$ is bounded by a constant, we know that there exist two constants $a,b>0$ such that for any $t\in (0,1)$,
\begin{equation}
\label{eq:est-Isym}
I_{sym}\left(a\left(\int_0^th^{N-1} dx_1 \right)^{ \frac{1}{p'} }\left(\int_{0}^t h^{N-2} dx_1 \right)^{\frac{1}{p}}\right)\geq b h(t)^{N-1}.
\end{equation}
We set
$$
m_0:=a\left(\int_0^{1}h^{N-1} dx_1 \right)^{ \frac{1}{p'}}\left(\int_{0}^{1} h^{N-2} dx_1 \right)^{\frac{1}{p}}.
$$
By change of variable,  an easy computation  which relays on \eqref{eq:est-Isym} and the fact that $h$ is   increasing, shows that 
\begin{align}
\label{eq:intsym}
\int_{0}^{m_0}\frac{dm}{I_{sym}(m)}&=\int_{0}^{1}\frac{\frac{d}{dt}\left(a\left(\int_0^th^{N-1}\, dx_1 \right)^{ \frac{1}{p'}}\left(\int_{0}^t h^{N-2}\, dx_1 \right)^{\frac{1}{p}}\right)}{I_{sym}\left(a\left(\int_0^th^{N-1}\, dx_1 \right)^{ \frac{1}{p'}}\left(\int_{0}^t h^{N-2}\, dx_1 \right)^{\frac{1}{p}}\right)}dt\\
\nonumber &\leq \frac{a}{b}\int_{0}^{1}\left(\frac{\int_0^t h^{N-2}\, dx_1 }{\int_0^t h^{N-1}\, dx_1 }\right)^\frac{1}{p}dt<\infty.
\end{align}

Let $\Gamma:=\partial\Om \cap \{1/4\leq x_1\leq 3/4\}$. We apply Theorem \ref{bgnr05}  to $\Om\cap \{1/4\leq x_1\leq 3/4\}$ and $\Gamma$ and get that $\inf_{\Om\cap \{1/3\leq x_1\leq 2/3\}}(u)>0$. In particular, 
\[\inf_{\Om\cap \{x_1=1/2\}}(u)=:c>0.\]
 Let $v$ be  the $p$-harmonic function on $\Om_{1/2}$ that verifies $ v =1$ on $\Om\cap\{x_1=1/2\}$ and a Robin boundary condition on $\partial\Om\cap\{x_1\leq 1/2\}$. By comparison principle,
\[u\geq c  v\text{ on }\Om_{1/2}.\]
Now, as $ v$ is the unique minimizer of $ w\mapsto \int_{\Om_{1/2}}|\nabla  w|^p\,dx+\beta\int_{\partial^e\Om_{1/2}}| w|^p\,d\Hm$ among functions that verify the constraint $ w=1$ on  $\partial^i \Om_1$, then it is rotationally symmetric; in particular its sublevel sets $\{ v<t\}$ belong to  $\mathcal{A}$ for $t\in (0,1)$. 
\par
Suppose then  $\inf v=0$, value which is reached approaching the cusp. For any $t\in ]0,1[$ let $\om_t:=\{ v<t\}$, and $f(t):=\int_{\om_t}|\nabla  v|\,dx$. With the same computations  as in the proof of Theorem \ref{bgnr01}, we find that
\[f(t)\leq  \beta^{1/p}  t P^e(\om_t)^{\frac{1}{p}} |\om_t|^{ \frac{1}{p'}}.\]
so that 
$$
I_{sym}(f(t)/ \beta^{1/p}  t)\leq f'(t).
$$ 
 The contradiction follows by integration, taking into account the summability property \eqref{eq:intsym}.
\end{proof}

\begin{remark}
{\rm
Theorem \ref{thm:t-cusp} particularly applies to  $h(t)=t^{\alpha}$ provided that $1\le \alpha<p$.  In fact,  $\alpha \in [1,p)$ is also a necessary condition to get the bound from below. 
\par
 Let indeed $u$ be  the $p$-harmonic function on $\Om_1$ equal to $1$ on $\partial^i\Om_1$ and with Robin boundary conditions on $\partial^e\Om_1$.  We claim that  $u$ extends continuously to $\overline{\Om_1}$.  Indeed  the continuity on the boundary $ \partial\Om_1 \setminus \{(0,0)\}$ is a consequence of boundary elliptic regularity  (see for example \cite[Theorem 4.4]{Nittka}). Let now  for $0<t\le 1$
\[\delta_t^+=\sup_{x\in\Om_1:x_1=t}(u)\qquad\text{and}\qquad \delta_t^-=\inf_{x\in\Om_1:x_1\leq t}(u).\]
 Notice that  $\delta_t^+$ is nondecreasing:  indeed, if $\delta_{t_1}^+>\delta_{t_2}^+$ for some $t_1<t_2$, we can consider the admissible function
$$
v(x):=
\begin{cases}
u(x)&\text{if }x_1\ge t_2\\
u(x) \wedge \delta_{t_2}^+ &\text{if }x_1<t_2
\end{cases}
$$
for which $E(v)<E(u)$, a contradiction. Let us denote with $\delta_0^\pm$ the limits of $\delta_t^\pm$ as $t\to 0$. To prove the continuity up to $(0,0)$, it suffices to show  $\delta_0^+=\delta_0^-$. Let $(t_i)_i$ be a sequence converging to $0$ such that 
\begin{equation}
\label{eq:choice}
\inf_{x\in\Om_1:x_1=t_i}u=:\delta_{t_i}\to \delta_0^-.
\end{equation}
Let 
$$
\tilde{\Om}_i:=\frac{\Om-(t_i,0)}{h(t_i)}\cap\{y:|y_1|\leq 1\},
$$
 which is a domain  that is $\mathcal{C}^1$-close to the cylinder $[-1,1]\times B_1^{N-1}$ where $B_1^{N-1}$ is the unit ball of $\R^{N-1}$. 
Let 
$$
u_i(x):=u(t_i+ h(t_i)  x_1,  h(t_i) x')
$$ 
be defined on $\tilde{\Om}_i$. Since the sequence $u_i$ is uniformly bounded and $\partial_\nu u_i\in [0,h(t_i)\beta]$,  then by boundary elliptic regularity  (see for instance the proof of \cite[Theorem 4.4]{Nittka}, thanks to which $u_i$ can be extended across the boundary satisfying an extension of our PDE) one can use Harnack inequality up to the boundary and infer that there exists a constant $C>0$ such that 
\[
\sup(u_i-\delta_{t_i+h(t_i)}^-)\leq C\inf (u_i-\delta_{t_i+h(t_i)}^-)\qquad \text{ on }\tilde{\Om}_i\cap\{x\in \R^N:|x_1|\leq 1/2\}.
\]
In view of \eqref{eq:choice}, this yields

\[
 \delta^+_{t_i}  \leq \delta_{t_i+h(t_i)}^-+C( \delta_{t_i} -\delta_{t_i+h(t_i)}^-)\underset{i\to\infty}{\longrightarrow}\delta_0^-,
\]
from which we get $\delta_t^+\to \delta_0^-$ as $t\to 0^+$, leading to the desired equality $\delta_0^+=\delta_0^-$.
\par
Suppose now that  $u$ satisfies a bound from below, that is  $(\delta_0^-\geq)\inf_{\Om_1}u>0$. Then for any $t>0$, and $\eps\in (0,1)$ to be fixed small enough later, consider the competitor
\[u_t(x)=\begin{cases}
u(x)&\text{ if }t\leq x_1\leq 1\\
\min\left((1-\eps)\delta_0^-+\left(\frac{2 x_1}{t}-1\right)\left(\delta_t^+-(1-\eps)\delta_0^-\right),u\right)&\text{ if }\frac{t}{2}\leq x_1\leq t\\
(1-\eps)\delta_0^-&\text{ if }0\leq x_1\leq \frac{t}{2}.
\end{cases}\]
 Notice that  for a small enough $t$ (depending on $\eps$) we have  $u_t \le u$ (as $u\ge \left( 1-\frac{\e}{2}\right) \delta_0^-$ for $x_1\le t$). The energy comparison gives
\[
\beta\int_{\partial^e\Om_t}\left[u^p-u_t^p\right]\,  d\Hm\leq \int_{\Om_t}\left[|\nabla u_t|^p-|\nabla u|^p\right]\,  dx
\]
which is simplified into
\[c_p\beta\eps  (\delta_0^-) ^p P^e(\Om_{t/2})\leq \frac{2^p}{t^p}\left|(\delta_t^+-\delta_0^-)+\eps \delta_0^-\right|^p|\Om_t|.     \]
Taking into account the expression of $h(t)=t^\alpha$, this gives
\[c_{p,\alpha}\beta^{\frac{1}{p}}t^{1-\frac{\alpha}{p}}\leq \frac{\delta_t^+-\delta_0^-}{\delta_0^-\eps^\frac{1}{p}}+\eps^{1-\frac{1}{p}}.\]
 Letting $t\to 0^+$, and since $\e$ is arbitrary, we conclude that $\alpha<p$.
}
\end{remark}

\begin{remark}
For the Laplace operator, in \cite[Examples 3.4 and 4.13]{BBC08} the authors discuss the cusps corresponding to $ h(t)=t^\alpha $. Their analysis is based on estimates of Green function and relies on a uniform boundary Harnack inequality. Although general functions $h$ are not considered \cite{BBC08}, it seems that for functions $h\in\mathcal{C}^0([0,1],\R_+)$ with $h(0)=0$, $h(t/2)>ch(t)$ for any $t>0$ and some $c>0$, the decomposition of the cusp $\Om$ given by into blocs $D_n=\{x\in\Om:t_{n+1}\leq x_1\leq t_n\}$ where the sequence $t_n$ is chosen such that each $D_n$ is close to a square, links their criteria to the summability of $\frac{\int_0^t h^{N-2}\,dx_1}{h(t)^{N-1}}$. This is not equivalent to our criteria of summability of $\sqrt{\frac{\int_{0}^{t} h^{N-2}\,dx_1}{\int_0^t h^{N-1}\,dx_1}}$, as may be seen from the example $h(t)=t^2\log(1/t)^{\alpha}$ for $\alpha\in (1,2]$. Presumably, the reason our criteria is weaker lies in our method that gives sharp inequalities as long as the sets $\{u<t\}$ are not too far from a minimizer of the relative perimeter, which is not what happens in this kind of cusps.
\end{remark}

We conclude the section with the proof of the technical Lemma \ref{lem:Isym}.

\begin{proof}[Proof of Lemma \ref{lem:Isym}]
We rely on the classification of constant mean curvature revolution surfaces in $\Rn$, meaning connected surfaces of $\Rn$ with constant mean curvature which are invariant for any isometry that fixes $e_1$. These surfaces are given by the revolution of a curve $(x(t),y(t))$ in $\R^2$ which, after parametrization by unit length, is defined by the equation
\[\begin{cases}x'=\cos(\sigma),\ y'=\sin(\sigma)\\
\sigma'=-NH+(N-1)y^{-1}\cos(\sigma).\end{cases}\]
Above,  $\sigma$ is the angle of the tangent vector and $H$ is the mean curvature. Moreover the quantity $T:=y^{N-1}\cos(\sigma)-H y^N$ is constant and the signs of $(H,T)$ fully classify the type of surface (see \cite[Proposition 2.4]{Ro03}). 
\par
 The main idea of the proof is the following : we consider a minimizing sequence $(A_k)$, and we change it into another sequence $(\tilde{A}_k)$ that is quasi-minimizing (it verifies the same constraint and $P^i(\tilde{A}_k)/P^i(A_k)$ is bounded) and that decomposes into the union of a set of the form $\Om_t$ and a set that is far from the axis of revolution, on which the analysis is simpler. 
\par
 Below we will  write $a\lesssim b$ when $a\leq Cb$ for a positive constant $C$ that may depend on $N$ and $h$ but not on the other quantities.
\par
 Given $t>0$, let us set
$$
m_t:=2P^e(\Om_t)^\frac{1}{p}|\Om_t|^{\frac{1}{p'}}.
$$

Consider $(A_k)_k$ a minimizing sequence for $I_{sym}( m_t)$. We lose no generality in replacing each $A_k$ with 
the minimizer of
\[\inf\left\{P^i(A),\ A\in\A\text{ s.t. }A\Delta A_k\subset \{x\in\Om:\text{dist}(x,\partial\Om)\geq \frac{1}{k},x_1\leq k\},\  |A|\geq |A_k|\right\}.\]
This problem is well posed and has a solution, by classical compactness arguments. Due to the classification of constant mean curvature revolution surfaces, we know that in the set $\{x\in\Om:\text{dist}(x,\partial\Om)> \frac{1}{k},x_1< k\}$, $A_k$ is smooth and $\partial A_k$ is a union of nonnegative constant mean curvature surfaces (when the normal vector is oriented outward). We denote $\partial ^e A_k: = \partial^* A_k \cap \partial \Om$ and  $\partial ^i A_k= \partial^* A_k \cap \Om$. The perimeter of $(A_k)_k$ is locally bounded so there is a subsequence of $(A_k)_k$ (that we denote with the same index) and a set $A\in\A$ such that $A_k\to A$ in $L^1_{\text{loc}}(\Om)$, with also a local Hausdorff convergence in $\Om$ due to the interior minimality of the sets $(A_k)$.
\begin{itemize}[label=\textbullet]
\item  Note that  $|A_k|$ and $P^e(A_k)$ are bounded from above and below. To prove that they are bounded from above we consider the projection $\Phi:(x_1,x')\in\Om\mapsto (h^{-1}(|x'|),x')$. It may be checked that 
$$\sup_{|x'|\geq h(1)}\Vert D\Phi\Vert<\infty \mbox{ and } \Hs\left(\partial^e A_k\setminus \Phi(\partial^iA_k) \right)=0$$
 so
\[
P^e(A_k)\leq P^e(\Om_1)+\Hs(\partial^{e}A_k\cap\{|x_1|\geq 1\})\lesssim P^e(\Om_1)+I_{sym}( m_t ).
\]
Similarly, 
$$
|A_k|^{1-\frac{1}{N}}\lesssim \Per(A_k)\lesssim P^e(\Om_1)+I_{sym}( m_t),
$$
by the classical isoperimetric inequality. In particular we know that $|A|<\infty$.
\item  We have  $I_{sym}( m_t)>0$. Suppose indeed that $I_{sym}( m_t)=0$. Then for any $\tau>0$ we get with the same reasonning on $\Om\cap\{x_1\geq \tau\}$ that $$P^e(A_k\cap\{x_1>\tau\})\leq \left(\sup_{|x'|\geq h(\tau)}\Vert D\Phi\Vert\right) P^i(A_k)\underset{k\to\infty}{\longrightarrow} 0$$
 for every $\tau$, so that $P^e(A_k)\to 0$, which contradicts the previous point.
\item $\partial A$ is a revolution surface of  constant mean curvature, such that its section 
\[A\cap \{(x_1,x_2,0, \dots, 0),x_1,x_2>0\}\]
is a union of convex sets that meet $\partial\Om$ with an angle less than $\frac{\pi}{2}$ as in  Figures \ref{figure_delaunay} and \ref{figure2}. Indeed, by regularity argument on minimal surfaces we know $\partial A$ is a union of rotationally symmetric surfaces with (nonnegative) constant mean curvature, which are moreover bounded (because $|A|<\infty$).  In view of  the classification of such surfaces,   the components of $A$ may only be the intersection of $\Om$ with
\begin{itemize}
\item[1)] a half-space $\{x_1\leq\lambda\}$ or a ball centered on $\{x'=0\}$ (observe that the boundary angle condition and the interior regularity implies that the ball meets $\partial\Om$).
\item[2)] the exterior of a catenoid or the interior of the loop of a nodoid (see the figure below).
\end{itemize}

\begin{figure}
\begin{center}
\includegraphics[scale=0.1]{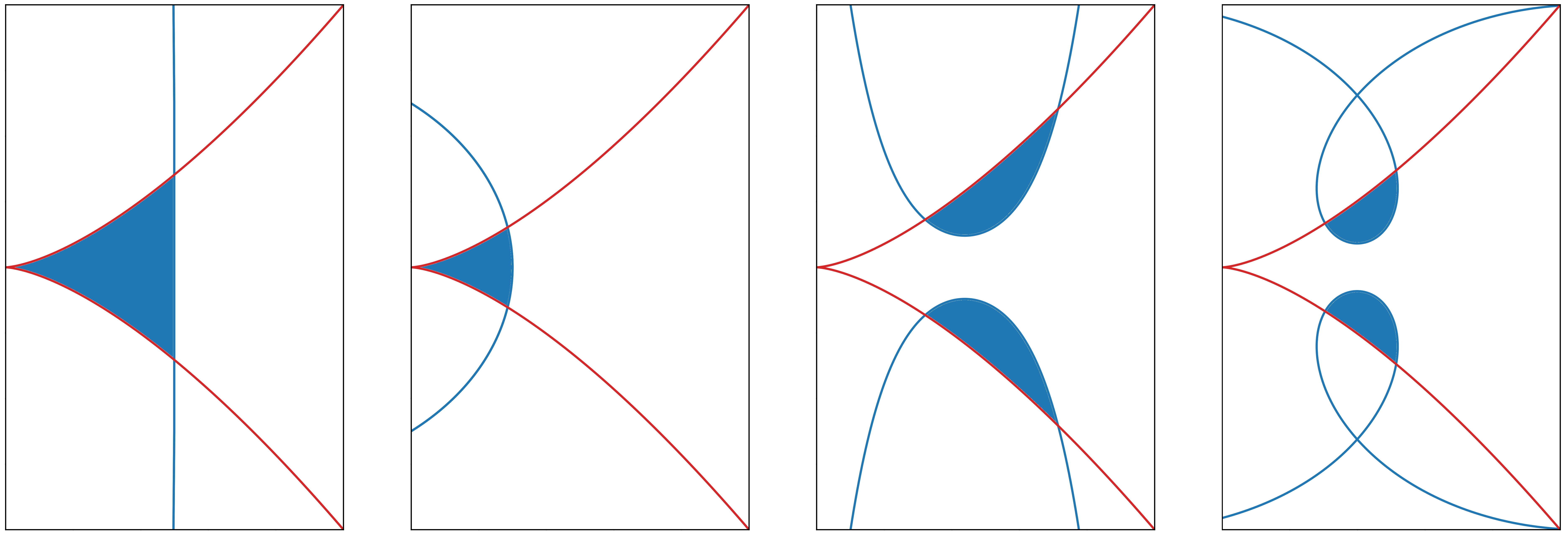}

\end{center}
\caption{\small Possible connected components of $A$, seen in the section $\R_+\times \R\times \{0\}^{N-2}$. From left to right: hyperplane, sphere, catenoid, nodoid.}
\label{figure_delaunay}
\end{figure}

In particular each component of the section is convex. Finally, let us check the angle condition; let $\tilde{A}$ be one of the connected components of $A$, and parametrize $\partial^i \tilde{A}\cap(\R_+^2\times \{0\}^{N-2})$ by a smooth curve $c:[0,T]\to \R_+^2$ that rotates clockwise where $c(0),c(T)\in\partial\Om$. There exists a sequence of curves $c_k:[0,T]\to \R_+^2$ that converges in $\mathcal{C}^1$ to $c_k$ such that $c_k((0,T))\subset\partial^i A_k$. Suppose the angle condition is not verified at $c(0)=(x_1,x_2)$ ($c(T)$ is handled similarly), meaning $c'(0)\cdot (1,h'(x_2))>0$. Then there is a small $\delta\in (0,1)$ such that $c'(t)\cdot (1,h'(x_2))>2\delta$ for $t\in (0,\delta)$, which implies that for any large enough $k$,
\[c_k'(t)\cdot (1,h'(x_2))>\delta,\ \forall t\in [0,\delta].\]
Let $p_k,q_k$ be the orthogonal projections of $c_k(0)$, $c_k(\delta)$ on the graph of $h$ (and similarly $p,q$ the projection of $c(0),c(\delta)$; notice $p=c(0)$) and let $T_k$ be the trapezoid formed by the graph of $h$ between $p_k$ and $q_k$, $[q_k,c_k(\delta)]$, $c_{k}([0,\delta])$, $[c_k(0),p_k]$. Let $T_k^{rev}\subset\R^N$ be the revolution of $T_k$ around the axis. We claim that the sequence $(A_k\cup T_k^{rev})_k$ still verifies the constraint (this is direct because the measure and exterior perimeter can only increase) while
\[\liminf_{k\to\infty} P^i(A_k\cup T_k^{rev})<\liminf_{k\to\infty}P^i(A_k)\]
which is a contradiction because $(A_k)$ is already a minimizing sequence. This last estimate is obtained from
\begin{align*}
P^i(A_k\cup T_k^{rev})-P^i(A_k)&\leq\Hs([p_k,c_k(0)]^{rev})+\Hs([q_k,c_k(\delta)]^{rev})-\Hs(c_{k}([0,\delta])^{rev})\\
&\underset{k\to\infty}{\longrightarrow}\Hs([q,c(\delta)]^{rev})-\Hs(c([0,\delta])^{rev})\\
&<0\text{ when }\delta\text{ is small enough.}
\end{align*}

\begin{figure}
\begin{center}
\includegraphics[scale=0.4]{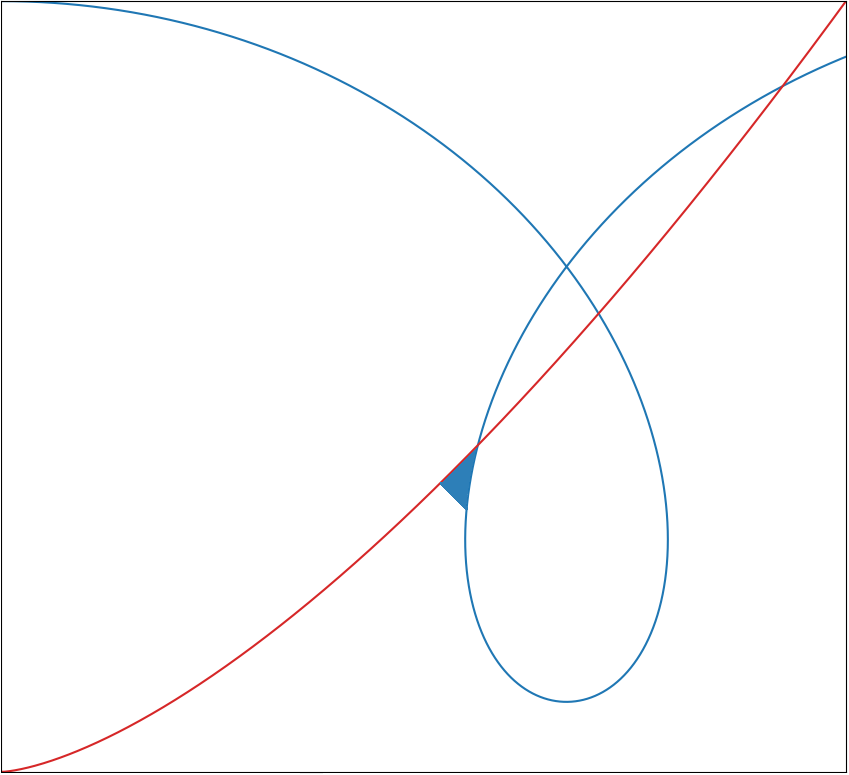}\\
\end{center}
\caption{The shaded region is added to $A_k$.}
\label{figure2}
\end{figure}
\end{itemize}

Let 
$$\Om'=\{(x_1,x')\in\Om:|x'|\leq \frac{1}{2}h(x_1)\}, \quad\Om''=\{(x_1,x')\in\Om:|x'|\leq \frac{1}{4}h(x_1)\}.$$
 We will now make a few modifications on the sequence $(A_k)_k$, to obtain a new sequence $\tilde{A}_k$ verifying the same constraints such that
\[
\tilde{A}_k=\Om_{t_k}\sqcup D_k,
\]
where $t_k>0$, $D_k\subset\Om\setminus\Om''$, and $P^i(\tilde{A}_k)\lesssim P^i(A_k)$.

 \begin{itemize}[label=\textbullet]
\item First modification: let $\delta_t>0$ be such that $P^i(\Om_{\delta_t})=I_{sym}(m_t)$ (note that we do not know a priori  how to compare $\delta_t$ and $t$). change $A_k$ into $A_k':=A_k\cup \Om_{ \delta_t }$. Notice that $A_k'$ still verifies the constraint and by choice of $\delta_t$, 
$$P^i(A_k')\leq 2I_{sym}( m_t )+o_{k\to\infty}(1).$$
\item Second modification: suppose there exists a point $( \tau ,0)\in  \partial A$ for some $ \tau>\delta_t $,  with $\tau$ chosen maximal. This means that for any large enough $k$, $\partial A_k$ contains 
$$
 S_k \cap \left\{x\in \Om\,:\, x_1\le k \text{ and } dist(x,\partial \Om)\ge \frac{1}{k}\right\},
$$
 where $S_k$ is either a hyperplane or a sphere going through $(\tau_k,0)$, where $\tau_k \to \tau$ (the hyperplane is orthogonal to the $x_1$ axes, while the center of the sphere is on the $x_1$-axes, and has an abscissa less that $\tau_k$). In this case we let $A_k'':=A_k'\cup\Om_{ \tau_k}$.  Again, we may assume $\tau_k$ to be chosen maximal.
\par
Consider the projection 
$$(x_1,x')\in  S_k\mapsto \left( \tau_k ,\frac{h( \tau_k )}{h(x_1)} x'\right)\in \{x_1= \tau_k \}.$$ 
Its differential is locally bounded on  $S_k$  because for any such $(x_1,x')\in  S_k$ one has $h(x_1)\geq \frac{1}{1+\Vert h'\Vert_{L^\infty}^2}h( \tau_k )$. Indeed the minimal value on $x_1$ here is given by the intersection of $S_k$ with $\partial\Om$ ; the boundary angle condition gives that $(\tau_k -x_1,-h(x_1))\cdot (1,h'(x_1))\leq 0$, so $h(  \tau_k )\leq h(x_1+h(x_1)h'(x_1))\leq \left(1+\Vert h'\Vert_{L^\infty}^{ 2 }\right)h(x_1)$. As a consequence,
\[P^i(A_k'')\lesssim I_{sym}( m_t )+o_{k\to\infty}(1).\]
\item Third modification: suppose $A\cap\{x_1> \tau\}\cap\Om''$ is not empty,  where $\tau$ is the same as in the previous point. Then $\partial A$ contains a part of catenoid or nodoid $C$ that  passes through $\partial\Om'$. We now make a disjunction of two cases.
\begin{itemize}
\item[Case 1.] The rightmost point of $C$ is in $\Om\setminus \Om''$. Then, for a large enough $k$ we know $\partial A_k$ contains a part of a catenoid or nodoid $C_k$ that approaches $C$ in $\mathcal{C}^1$ such that its rightmost point is in $\Om\setminus\Om''$, and it meets $\partial \Om''$ at some $(T_k,\frac{1}{4}h(T_k))$. Then we let
\[\tilde{A}_k=A_k''\cup \Om_{T_k}.\]
Again $\tilde{A}_k$ verifies the constraint and with the same horizontal projection argument as in the previous point
$$
\Hm(\partial^i \Om_{T_k} )\lesssim \Hm(\partial^i \Om_{T_k} \cap (\Om\setminus \Om''))\lesssim P^i(A_k)$$
so $P^i(\tilde{A_k})\lesssim P^i(A_k)$, where the second inequality is obtained by projection of $C_k$ on the annulus $\left\{(T_k,x'):\frac{1}{4}h(T_k)\leq |x'|\leq h(T_k)\right\}$ through the application $\psi:(x_1,x')\in C_k\mapsto \left( T_k ,\frac{h(T_k)}{h(x_1)} x'\right)\in \{x:x_1=T_k\}$.
 \item[Case 2.] Or the rightmost point of $C$ is reached in $\Om''$, so the rightmost point of $C_k$ is reached in $\Om'$. We then denote $T_k$ its abscissa and we let
\[\tilde{A}_k=A_k''\cup \Om_{T_k}\]
as previously. With the same projection, $P^i(\tilde{A_k})\lesssim P^i(A_k)$.
\end{itemize}

\end{itemize}

After having performed the previous modifications, we went from $A_k$ to $\tilde{A}_k$ such that
\[
\tilde{A}_k=\Om_{t_k}\sqcup D_k,
\]
where $t_k\geq\delta_t$, 
\[
D_k\subset\Om\setminus\Om'',
\] 
and
$$
P^i(\tilde{A}_k)\lesssim I_{sym}(m_t)+o_{k\to \infty}(1).
$$

Notice that the ``classical'' relative isoperimetric inequality applies to $D_k$. If we consider indeed the projection $\pi:\Om\setminus\Om''\mapsto \partial\Om$ such that
\[
\pi(x_1,x'):=\left(x_1,\frac{x'}{|x'|}h(x_1)\right)
\]
then its differential is bounded on $\Om\setminus\Om''$. Moreover $\partial^eD_k \subset \pi(\partial^i D_k)$ up to a $\Hs$-negligible set, from which we get
\[P^e(D_k)+|D_k|^{1-\frac{1}{N}}\lesssim P^i(D_k).\]
As well, for $k$ large enough
\begin{equation}
\label{eq:ineq-fin}
P^i(\Om_{t_k}) \lesssim   P^i(\tilde{A}_k) \qquad\text{and}\qquad P^{i}(\tilde{A}_k) \lesssim   I_{sym}( m_t).
\end{equation}
If $t_k\geq t$  then inequality \eqref{eq:ineqL5} follows. Assume that $t_k<t$. Then
\begin{align*}
2P^e(\Om_t)^{\frac{1}{p}}|\Om_t|^{\frac{1}{p'}}&= m_t\lesssim  P^e(\tilde{A}_k)^{\frac{1}{p}}|\tilde{A}_k|^{ \frac{1}{p'}}\\
&=(P^e(\Om_{t_k})+P^e(D_k))^{\frac{1}{p}}(|\Om_{t_k}|+|D_k|)^{ \frac{1}{p'}}\\
&\leq\left(P^e(\Om_t)+cP^i(D_k)\right)^{\frac{1}{p}}\left(|\Om_t|+cP^i(D_k)^{\frac{N}{N-1}}\right)^{ \frac{1}{p'}}\\
\end{align*}
for some constant $c>0$. We thus infer
$$
P^i(D_k)\gtrsim |\Om_t|^{1-\frac{1}{N}}\wedge P^e(\Om_t).
$$ 
Since  in view of \eqref{eq:ineq-fin}
$$
P^i(D_k) \le P^i(\Om_{t_k})+P^i(\tilde A_k) \lesssim P^i(\tilde A_k) \lesssim I_{sym}( m_t)=: \sigma_t P^{i}(\Om_t)
$$
 for some $\sigma_t>0$, we deduce 
 $$
 \sigma_t \gtrsim \frac{|\Om_t|^{1-\frac{1}{N}}\wedge P^e(\Om_t)}{P^i(\Om_t)}.
 $$
The conclusion follows if we estimate $\sigma_t$ from below. Since $P^i(\Om_t)=o_{t\to 0}\left(P^e(\Om_t)\right)$ it is enough to bound from below $\frac{|\Om_t|^{1-\frac{1}{N}}}{P^i(\Om_t)}$. Using the convexity of $h^{N-1}$ we deduce \[|\Om_t|=\int_{0}^{t}h^{N-1}  (x_1)\,dx_1 \geq \frac{1}{2}\frac{h(t)^{2N-2}}{\frac{d}{dt}h^{N-1} (t) }=\frac{h(t)^N}{(2N-2)h'(t)},\]
so we get $\sigma_t\gtrsim 1\wedge \frac{1}{h'(t)^{1-\frac{1}{N}}}$, which ends the proof.
\end{proof}

\section{Further remarks}
\noindent{\bf More general operators.}
The results of the paper  extend naturally to more general elliptic problems with Robin boundary conditions, which are not of energy type.  
So let $\Om\sq \R^N$ be a bounded, connected, open  set, with rectifiable boundary, such that $\Hm(\partial \Om)<+\infty$. 

Let 
  $${\mathcal A} : \Om \times \R^N \ra \R^N$$
  $${\mathcal B} : \partial \Om \times \R \ra \R$$
  $$\psi : \partial\Om \ra \R$$
  be  three continuous functions such that for some $0< \alpha _1\le \alpha_2$,  and every  $x\in \Om, y \in \partial \Om, z \in \R^N, v \in \R$,   
  $$ \alpha_1 |z|^p\le z{\mathcal A} (x,z) , |{\mathcal A} (x,z) |\le \alpha_2|z|^{p-1},$$
  $$|{\mathcal B} (y,v)| \le \psi(y) |v|^{p-1}.$$
  Assume moreover that  for every  $x\in \Om, y \in \partial \Om, z_1, z_2 \in \R^N, v_1, v_2 \in \R$
  $$(z_1-z_2)( {\mathcal A} (x,z_1)-{\mathcal A} (x,z_2)) \ge 0,$$
  $$(v_1-v_2)({\mathcal B} (y,v_1)-{\mathcal B} (y,v_2)) \ge 0.$$
  We consider the (formal) problem
  \begin{equation}
\label{pden}
\begin{cases}
-div ({\mathcal A} (x, \nabla u)) =f&\hbox{in }\Om\\
\ds {\mathcal A} (x, \nabla u) \cdot {\bf n}+{\mathcal B(x,u)}=0&\hbox{on }\partial\Om.
\end{cases}
\end{equation}
 We do not develop around the question of existence of a (weak) solution in this framework and refer to the paper by R. Nittka \cite{Ni13} for an introduction to general Robin problems in Lipschitz sets, to \cite{HKM93} for the analysis of ${\mathcal A}$-superharmonic functions and to \cite{BG16} for details on the framework of domains with rectifiable boundary. 

Let $f \in L^{p'}(\Om)$, $f\ge 0$, $f\not=0$.  Assume that a weak solution exists in our nonsmooth context, namely that there exists $u \in W^{1,p} (\Om)$ such that $\forall v \in W^{1,p} (\Om)$
$$ \int _\Om   {\mathcal A}(x, \nabla u) \nabla v dx + \int_{\partial \Om}[ {\mathcal B} (x,u^+)v^+ + {\mathcal B} (x,u^-)v^- ]d \Hm = \int_\Om f v dx.$$

The fact that $u \ge 0$ is a consequence of the properties of ${\mathcal A}, {\mathcal B}$ and can be noticed by testing the equation with $u \wedge 0$. 
Taking $u \vee t$ as test function  for $t >0$, and using again the properties of ${\mathcal A}, {\mathcal B}$ and $\psi$ 
one gets directly an inequality similar to \eqref{eq3.1}
\begin{equation}\label{eq3.1b}
\int_{\om_t}|\nabla u|^p \,dx\leq \frac{|\psi|_\infty}{\alpha_1} t^pP^e(\om_t),
\end{equation}
which is the key ingredient of our results. The proofs of Theorems \ref{bgnr01} and \ref{bgnr05} can be continued from this point on.

\medskip
\noindent{\bf More general open sets.}
It could be possible to deal with more general open sets,  removing the rectifiability hypothesis, but some important drawbacks occur. However, this removes the Robin problem from Sobolev spaces, the natural context being the one of  the SBV functions, the very first question being its well posedness. The traces  can occur only on the rectifiable part of the boundary; on the purely non-rectifiable part, the functions that we consider do not have jumps ! In other words, they do not behave as boundary points, even though they belong to the topological boundary.  
\medskip

\noindent
{\bf Acknowledgments.} D.B. and M.N. were supported by  ANR SHAPO (ANR-18-CE40-0013).
\bibliographystyle{mybst}
\bibliography{References}

\end{document}